\documentclass[12pt]{amsart}

\usepackage[matrix,arrow,curve,frame]{xy}
\usepackage{amsmath,amsthm,amssymb,enumerate}
\usepackage{latexsym}
\usepackage{amscd}
\usepackage[colorlinks=false]{hyperref}
\usepackage{euscript}

\newtheorem{thm}[subsection]{Theorem}

\newtheorem{cor}[subsection]{Corollary}
\newtheorem{lemma}[subsection]{Lemma}

\newtheorem{remark}[subsection]{Remark}

\theoremstyle{definition}

\numberwithin{equation}{section}

\def\cQ{{\bf Q}}

\def\cA{{\cal A}}

\def\ra{\rightarrow}

\def\cA{{\mathcal A}}

\def\cQ{{\mathcal Q}}

\def\gl{{\mathfrak l}}

\def\gs{{\mathfrak s}}

\newfont{\german}{eufm10}

\begin{document}
\pagestyle{plain}

\title{the global sections of the chiral de Rham complex on a Kummer surface}

\author{Bailin Song}

\thanks{The author is supported in part by NSFC No. 11101393 }

\address{Wu Wen-Tsun Key Laboratory of Mathematics, USTC, Chinese Academy of Sciences and Department of Mathematics, University of Science and Technology of China, Hefei, Anhui, 230026, P.R. China}
\email{bailinso@ustc.edu.cn}

\begin{abstract}  The chiral de Rham complex is a sheaf of vertex algebras $\Omega^{\text{ch}}_M$ on any nonsingular algebraic variety or complex manifold $M$, which contains the ordinary de Rham complex as the weight zero subspace. We show that when $M$ is a Kummer surface, the algebra of global sections is isomorphic to the $N=4$ superconformal vertex algebra with central charge $6$. Previously, $\mathbb{C}\mathbb{P}^n$ was the only manifold where a complete description of the global section algebra was known. \end{abstract}


\maketitle
%\tableofcontents
\section{Introduction}
In 1998, Malikov, Schechtman, and Vaintrob \cite{MSV} constructed a sheaf of vertex algebras $\Omega^{\text{ch}}_M$ known as the \textsl{chiral de Rham complex} on any nonsingular algebraic variety or complex manifold $M$. It has a conformal weight grading by the non-negative integers, and the weight zero piece coincides with the ordinary de Rham sheaf. This construction has substantial applications to mirror symmetry and is related to stringy invariants of $M$ such as the elliptic genus \cite{Bor}. According to \cite{Kap}, on any Calabi-Yau manifold, the cohomology $H^*(M, \Omega^{\text{ch}}_M)$ can be identified with the infinite-volume limit of the half-twisted sigma model model defined by E. Witten. 

For simplicity of notation, we shall denote the sheaf $\Omega^{\text{ch}}_M$ by $\cQ = \cQ_M$, and we shall denote its global section algebra $H^0(M, \Omega^{\text{ch}}_M)$ by $\cQ(M)$. Certain geometric structures on $M$ give rise to interesting substructures of $\cQ(M)$. For example, if $M$ is a Calabi-Yau manifold, $\cQ(M)$ contains a topological vertex algebra structure, or equivalently, an $N=2$ superconformal structure \cite{MSV}. If $M$ is hyperkahler, $\cQ(M)$ contains an $N=4$ superconformal structure with $c=3d$ where $d$ is the complex dimension of $M$ \cite{BHS,H}. Describing the vertex algebra structure of $\cQ(M)$ by giving generators and operator product expansions is a difficult problem; until now, $\mathbb{C}\mathbb{P}^n$ was the only nontrivial manifold where such a description was known \cite{MS}. In this paper, we shall give a complete description of $\cQ(M)$ in the case where $M$ is a Kummer surface. Since $M$ is hyperkahler it contains an $N=4$ superconformal algebra with $c=6$, and our main result is that $\cQ(M)$ is \textsl{isomorphic} to this $N=4$ algebra.

 Our main tool is a filtration $\{\cQ_n\}$ on $\cQ$ for any $M$ which is a \textsl{good increasing filtration} in the sense of \cite{LiII}, for which the associated graded object $\text{gr}(\cQ)=\bigoplus \cQ_n/\cQ_{n-1}$ is a sheaf of supercommutative graded rings on $M$. There is another, more refined filtration of $\text{gr}(\cQ)$ whose associated graded object $\text{gr}^2(\cQ)$ is also a sheaf of supercommutative graded rings and is closely related to classical tensor bundles on $M$. In the case of Kummer surfaces, the ring of global sections of $\text{gr}^2(\cQ)$ can be described explicitly using techniques developed in our previous work \cite{LSSI, LSSII} on the interaction between jet schemes, arc spaces, and classical invariant theory. We shall prove that for a Kummer surface, the ring of global sections of $\text{gr}^2(\cQ)$ it is exactly the $\gs\gl_2[t]$-invariant subalgebra of the sheaf restricted to one point; see Corollary \ref{thm.inv}. This result can be regarded as the arc space analogue of a theorem of S. Kobayashi \cite{Ko} which says that holomorphic sections of the bundle given by tensors of tangent and cotangent bundles are exactly the set of $SU_2$ invariants of a fiber of the bundle. It can be expected that a similar result will hold for the global sections of $\text{gr}^2(\cQ)$ on other Calabi-Yau manifolds. Finally, using the relationship between global sections of $\cQ$ and $\text{gr}^2(\cQ)$, we find a minimal strong generating set for the vertex algebra of global sections of $\cQ$ on Kummer surfaces; see Theorem \ref{thm.main}.

The plan of this paper is following. In Section \ref{sec.va} we briefly introduce vertex algebras and the chiral de Rham complex. In Section \ref{sec.ag} the associated graded sheaves of the chiral de Rham complex are constructed, 
and the relations of their global sections are studied. In Section \ref{sec.invt}, the ring of $\mathfrak{sl}_2[t]$ invariants of 
the sheaf $\text{gr}^2(\cQ)$ restricted to a point is studied and a linear basis of this ring is constructed. In Section \ref{sec.global}, the global sections of $\text{gr}^2(\cQ)$ and finally, the global sections of $\cQ$ on Kummer surfaces are determined.

\section{Vertex algebras}\label{sec.va}

First we define vertex algebras, which have been discussed from various points of view in the literature (see for example \cite{B,FLM,Kac}). We will follow the formalism developed in \cite{LZ} and partly in \cite{LiI}. Let $V=V_0\oplus V_1$ be a super vector space over $\mathbb{C}$, and let $z,w$ be formal variables. Let $\text{QO}(V)$ be the space of linear maps $$V\ra V((z))=\{\sum_{n\in\mathbb{Z}} v(n) z^{-n-1}|
v(n)\in V,\ v(n)=0\ \text{for} \ n>\!\!>0 \}.$$ Elements $a\in \text{QO}(V)$ can be
represented as power series
$$a=a(z)=\sum_{n\in\mathbb{Z}}a(n)z^{-n-1}\in \text{End}(V)[[z,z^{-1}]].$$ Each $a\in
\text{QO}(V)$ has the form $a=a_0+a_1$ where $a_i:V_j\ra V_{i+j}((z))$ for $i,j\in\mathbb{Z}/2\mathbb{Z}$, and we write $|a_i| = i$. There are nonassociative bilinear operations
$\circ_n$ on $\text{QO}(V)$, indexed by $n\in\mathbb{Z}$, which we call the $n^{\text{th}}$ circle
products. For homogeneous $a,b\in \text{QO}(V)$, they are defined by
$$ a(w)\circ_n b(w)=\text{Res}_z a(z)b(w)\ \iota_{|z|>|w|}(z-w)^n-
(-1)^{|a||b|}\text{Res}_z b(w)a(z)\ \iota_{|w|>|z|}(z-w)^n.$$
Here $\iota_{|z|>|w|}f(z,w)\in\mathbb{C}[[z,z^{-1},w,w^{-1}]]$ denotes the
power series expansion of a rational function $f$ in the region
$|z|>|w|$. We usually omit $\iota_{|z|>|w|}$ and just
write $(z-w)^{-1}$ to mean the expansion in the region $|z|>|w|$,
and write $-(w-z)^{-1}$ to mean the expansion in $|w|>|z|$. For $a,b\in \text{QO}(V)$, the following formal power series identity is known as the \textsl{operator product expansion} (OPE) formula.
 \begin{equation}\label{opeform} a(z)b(w)=\sum_{n\geq 0}a(w)\circ_n
b(w)\ (z-w)^{-n-1}+:a(z)b(w):. \end{equation}
Here $:a(z)b(w):\ =a(z)_-b(w)\ +\ (-1)^{|a||b|} b(w)a(z)_+$, where $a(z)_-=\sum_{n<0}a(n)z^{-n-1}$ and $a(z)_+=\sum_{n\geq
0}a(n)z^{-n-1}$. Often we write
$$a(z)b(w)\sim\sum_{n\geq 0}a(w)\circ_n b(w)\ (z-w)^{-n-1},$$ where
$\sim$ means equal modulo the term $:a(z)b(w):$, which is regular at $z=w$. 

Note that $:a(w)b(w):$ is a well-defined element of $\text{QO}(V)$, called the \textsl{Wick product} of $a$ and $b$, and it
coincides with $a\circ_{-1}b$. The other negative products are given by
$$ n!\ a(z)\circ_{-n-1}b(z)=\ :(\partial^n a(z))b(z):,\qquad \partial = \frac{d}{dz}.$$
For $a_1(z),\dots ,a_k(z)\in \text{QO}(V)$, the iterated Wick product is defined to be
\begin{equation}\label{iteratedwick} :a_1(z)a_2(z)\cdots a_k(z):\ =\ :a_1(z)b(z):,\qquad b(z)=\ :a_2(z)\cdots a_k(z):.\end{equation}
We often omit the formal variable $z$ when no confusion can arise.

$\text{QO}(V)$ is a nonassociative algebra with the operations $\circ_n$. A subspace $\cA\subset \text{QO}(V)$ containing $1$ which is closed under all $\circ_n$ will be called a \textsl{quantum operator algebra} (QOA). A subset $S=\{a_i|\ i\in I\}$ of $\cA$ is said to \textsl{generate} $\cA$ if every $a\in\cA$ can be written as a linear combination of nonassociative words in the letters $a_i$, $\circ_n$, for $i\in I$ and $n\in\mathbb{Z}$. We say that $S$ \textsl{strongly generates} $\cA$ if every $a\in\cA$ can be written as a linear combination of words in the letters $a_i$, $\circ_n$ for $n<0$. Equivalently, $\cA$ is spanned by $$\{ :\partial^{k_1} a_{i_1}\cdots \partial^{k_m} a_{i_m}:| \ i_1,\dots,i_m \in I,\ k_1,\dots,k_m \geq 0\}.$$ We say that $a,b\in \text{QO}(V)$ \textsl{quantum commute} if $(z-w)^N [a(z),b(w)]=0$ for some $N\geq 0$. Here $[,]$ denotes the super bracket. A \textsl{commutative} QOA is a QOA whose elements pairwise quantum commute. This notion is the same as the notion of a \textsl{vertex algebra}, and we will use these notions interchangeably.

%\section{Associated graded sheaves of the chiral de Rham complex}

\subsection{Filtered vertex algebras}\label{sub.1} We recall Li's notion of a \textsl{good increasing filtration} \cite{LiII} of a vertex algebra $V$. This will be a filtration by subspaces
$$\cdots\subset V_{-1}\subset  V_{0}\subset  V_{1}\subset  V_{2}\subset  \cdots, \quad \bigcup V_n=V$$
such that $1\in V_0$ and for $a\in V_n$, $b\in V_m$,
$$a\circ_k b\in V_{n+m},\quad\quad\text{for } k<0,$$
$$a\circ_k b\in V_{n+m-1},\quad\text{for } k\geq 0.$$
A \textsl{filtered vertex algebra} is a vertex algebra $V$ equipped with a good increasing filtration $\{V_n\}$.
Let $$\phi_n: V_n\to V_n/V_{n-1}$$ be the quotient map. Then 
$$\text{gr}(V)=\bigoplus (\text{gr}(V))_n=\bigoplus V_n/V_{n-1}$$
is a \textsl{vertex Poisson algebra}, i.e. a graded, associative, super-commutative algebra $\mathcal A$ equipped with a derivation $\partial$, and a family of derivations $ a\circ_k$ for each $k\geq 0$ and $a \in \mathcal A$. For $\text{gr}(V)$, the associative product is given by 
$$\phi_n(a)\phi_m(b)=\phi_{n+m}(:ab:)\in V_{n+m}\slash V_{n+m-1}, \text{ for } a\in V_n, b\in V_m.$$
The derivation $\partial$  is given by 
$$\partial\phi_n( a)=\phi_n(\partial a), \text{ for } a \in V_n.$$ 
The derivations of elements $\phi_n(a)$ in  $V_n/V_{n-1}$ are given by 
$$ \phi_n(a)\circ_k\phi_m(b)=\phi_{n+m-1}(a\circ_k b)\in V_{n+m-1}\slash V_{n+m-2},  \text{ for } a\in V_n, b\in V_m, k\geq 0.$$
A $\mathbb Z_{\geq 0}$ graded, associative, super-commutative algebra  equipped with a derivation $\partial$ of degree zero is called $\partial$-ring~\cite{LL}. A $\partial$-ring is just an \textsl{abelian} vertex algebra, that is, a vertex algebra in which all the nonnegative circle products  are zero. If $V_n=0$ for $n<0$, then $\text{gr}(V)$ is $\mathbb Z_{\geq 0}$ graded and it is a $\partial$-ring.

\subsection{Chiral de Rham complex}
The chiral de Rham complex~\cite{MSV} is a sheaf of vertex algebras defined on any smooth manifold $M$ in either the algebraic, complex analytic, or $C^{\infty}$ categories. In this paper we work exclusively in the complex analytic setting, although the filtrations we introduce are valid in all the above settings.
Let $\Omega_N$ be the tensor product of $N$ copies of the $\beta\gamma-bc$ system. It has $2N$ even generators $\beta^1(z),\cdots ,\beta^N(z),\gamma^1(z),\cdots,\gamma^N(z)$ and $2N$ odd generators $b^1(z),\cdots ,b^N(z), c^1(z),\cdots,c^N(z)$.
Their nontrivial OPEs are $$\beta^i(z) \gamma^j(w)\sim \frac {\delta^i_j}{z-w},\quad b^i(z) c^j(w)\sim \frac {\delta^i_j}{z-w} .$$
The following four fields in $\Omega_N$
$$L(z)=\sum (:\beta^i(z)\partial\gamma^i(z):-:b^i(z)\partial c^i(z):), \quad
Q(z)= \sum:\beta^i(z)c^i(z):,$$ 
$$J(z)=-\sum:b^i(z)c^i(z):, \quad
G(z)=\sum:b^i(z)\partial\gamma^i(z):,$$
give $\Omega_N$ the structure of a topological vertex algebra \cite{LZ}.

For an open subset $U\subset \mathbb C^N$, let $\mathcal O(U)$ be the space of complex analytic functions on $U$. 
Let $\gamma^1,\dots,\gamma^N$ be coordinates on $\mathbb C^N$. We may regard
$\mathbb C[\gamma^1,\dots,\gamma^N]\subset \mathcal O(U)$ as a subspace of $\Omega_N$
by identifying $\gamma^i$ with  $\gamma^i(z)\in \Omega_N$. As a linear space, $\Omega_N$ has a $\mathbb C[\gamma^1,\dots,\gamma^N]$ module structure. We define $\cQ(U)$ to be the localization of $\Omega_N$ on 
$U$, $$\cQ(U)=\Omega_N\otimes_{\mathbb C[\gamma^1,\dots,\gamma^N]}\mathcal O(U).$$
Then $\cQ(U)$ is the vertex algebra generated by $\beta^i(z), b^i(z), c^i(z)$ and $f(z)$, $f\in \mathcal O(U)$. These satisfy the nontrivial OPEs
$$\beta^i(z)  f(w)\sim \frac {\frac{\partial f}{\partial \gamma ^i}(z)}{z-w},\quad b^i(z) c^j(w)\sim \frac {\delta^i_j}{z-w},$$ as well as the normally ordered relations $$:f(z)g(z):\ =fg(z), \text{ for }  f, g \in \mathcal O(U).$$
 Note that $\cQ(U)$ is spanned by the elements
\begin{equation}\label{eq.span}
:\partial^{k_1}\beta^{i_1}(z)\cdots \partial^{k_s}\beta^{i_s}(z)\partial^{l_1}b^{j_1}(z)
\cdots \partial^{l_t}b^{j_t}\partial^{m_1}c^{r_1}(z)
 \cdots\partial^{n_1}\gamma^{s_1}(z)\cdots f(\gamma)(z): , \quad f(\gamma)\in \mathcal O(U) .
 \end{equation}
Now let $\tilde \gamma^1,\dots, \tilde \gamma^N$ be another set of coordinates on $U$, with
$$\tilde \gamma^i=f^i(\gamma^1,\dots,\gamma^N), \quad \gamma^i=g^i(\tilde \gamma^1,\dots,\tilde \gamma^N).$$
We have the following coordinate transformation equations for the generators.
\begin{align}\label{chi.coo}
\partial \tilde \gamma^i(z)&=:\frac{\partial f^i}{\partial \gamma^j}(z)\partial \gamma^j(z):\,, \nonumber \\
\tilde b^i(z)&=:\frac{\partial g^i}{\partial \tilde \gamma^j}(g(\gamma))b^j: \nonumber\,, \\
\tilde c^i(z)&=:\frac{\partial f^i}{\partial \gamma^j}(z)c^j(z):\,,\\
\tilde \beta^i(z)&=:\frac{\partial g^i}{\partial \tilde \gamma^j}(g(\gamma))(z)\beta^j(z):
+::\frac{\partial}{\partial \gamma}(\frac{\partial g^i}{\partial \tilde \gamma^j}(g(\gamma)))(z)c^k(z):b^j(z):\,.\nonumber
\end{align}
 For a given complex manifold $M$, let 
 $\mathcal U$ be the set of the open sets of $M$ which are isomorphic to open sets of $\mathbb C^n$. The correspondence $U\to \cQ(U)$  for $U\in \mathcal U$, defines a vertex algebra sheaf $\cQ = \cQ_M$ called the \textsl{chiral de Rham complex} on $M$.

When $M$ is Calabi-Yau, i.e. $c_1(TM)=0$, the four fields $L(z),J(z), Q(z), G(z)$
 are globally defined on $M$, giving $\cQ(M)$ the structure of  a topological vertex algebra.

\section{Associated graded sheaves of the chiral de Rham complex}\label{sec.ag}

For $U\in \mathcal U$, we consider the filtration
$$\cQ_0(U)\subset \cQ_1(U)\subset\cdot\cdot\cdot\cdot\subset\cQ_n(U)\subset\cdot\cdot\cdot, \quad \cQ(U)=\bigcup \cQ_n(U).$$
Here $\cQ_n(U)$ is spanned by the elements with only at most $n$ copies of $\beta$ and $b$ , i.e. the elements in equation~\eqref{eq.span} with $s+t\leq n$.

For each $n\geq 0$, let  $$\phi_n:\cQ_n(U )\to \cQ_n(U)/\cQ_{n-1}(U).$$ be the projection.
\par
 For $a\in \cQ_k(U)$, $b\in \cQ_l(U)$, we have $1\in \cQ_{0}(U)$ and
\begin{align*}
&a\circ_n b\in \cQ_{k+l}(U), \quad &n<0;\\
&a\circ_n b\in \cQ_{k+l-1}(U), \quad &n\geq 0.
\end{align*}
Then $\{\cQ_n(U)\}$ is a good increasing filtration of $\cQ(U)$ and the associated graded object
$$\text{gr}(\cQ(U))=\bigoplus \cQ_n(U)/\cQ_{n-1}(U)$$
is $\partial$-ring as in Section~\ref{sub.1}. 
As a $\partial$-ring, $\text{gr}(Q(U))$ is generated by 
$$\beta^i=\phi_1(\beta^i(z)), \quad b^i=\phi_1(b^i(z)),\quad c^i= \phi_0(c^i(z)),$$ 
and $f(\gamma)=\phi_0(f(\gamma)(z))$, $f(\gamma) \in \mathcal O(U)$.

 The coordinate transformation equations ~(\ref{chi.coo}) for $\cQ(U)$ clearly preserve the filtration. 
 So for $U\in \mathcal U$ , $U\to \cQ_n(U)$ gives a sheaf $\cQ_n$ on $M$ and $\{\cQ_n\}$ is a filtration of the sheaf $\cQ$. 
 The correspondence $U\to \text{gr}(\cQ(U))$ with $U\in \mathcal U$, gives a sheaf of $\partial$-rings on the manifold $M$, 
 which we denote by $\text{gr}(\cQ)$. Obviously $$\text{gr}(\cQ)=\bigoplus_n \cQ_n/\cQ_{n-1}.$$

For the sheaf $\text{gr}(\cQ)$, the coordinate transformation equations of $\beta, \gamma, b, c $ are
\begin{align}\label{chi.co2}
\partial^n \tilde \gamma^i&=\partial^{n-1}(\frac{\partial f^i}{\partial \gamma^j}\partial \gamma^j),\nonumber\\
\partial^n \tilde b^i&=\partial^n(\frac{\partial g^i}{\partial \tilde \gamma^j}(g(\gamma))b^j),\nonumber\\
\partial^n \tilde c^i&=\partial^n(\frac{\partial f^i}{\partial \gamma^j}c^j),\\
\partial^n \tilde \beta^i&=\partial^n(\frac{\partial g^i}{\partial \tilde \gamma^j}(g(\gamma))\beta^j)+
\partial^n(\frac{\partial}{\partial \gamma}(\frac{\partial g^i}{\partial \tilde \gamma^j}(g(\gamma)))c^kb^j).\nonumber
\end{align}
 The only difference between these equations and those of the chiral de Rham sheaf is that the Wick product 
 is replaced by the ordinary product in an associated supercommutative algebra.  The globally defined operator $\partial=L\circ_0$ on $\cQ$ gives rise to a globally defined operator on $\text{gr}(\cQ)$, also denoted by $\partial$, satisfying
 $$ \phi_n(\partial a)=\partial \phi_n(a), \quad a\in \cQ_n(M).$$
 In particular, the space of global sections of $\text{gr}(\cQ)$ is a $\partial$-ring.

\subsection{A refined graded sheaf} Notice that  in the coordinate transformation equations \eqref{chi.co2} of $\text{gr}(\cQ(U))$, 
the equation for $\partial^n \beta$ is nonclassical because of the term
$\partial^n(\frac{\partial}{\partial \gamma}(\frac{\partial g^i}{\partial \tilde \gamma^j}(g(\gamma)))c^kb^j)$. 
To omit this term from the coordinate transformation equations, we construct a filtration of the sheaf $\text{gr}(\cQ)$.
First, $\cQ_n(U)/\cQ_{n-1}(U)$ is spanned by
\begin{equation} \label{bas:a}
a=\partial^{k_1}\beta^{i_1}\cdot\cdot\cdot\partial^{k_s}\beta^{i_s}
\partial^{k_{s+1}}b^{i_{s+1}}\cdot\cdot\cdot \partial^{k_{n}}b^{i_{n}}\partial^{m_1}c^{r_1}
 \cdot\cdot\cdot\partial^{n_1}\gamma^{s_1}\cdot\cdot\cdot \partial^{n_t}\gamma^{s_t}f(\gamma),\quad f(\gamma)\in \mathcal O(U).
 \end{equation}
Define $\text{deg}_{\beta}(a)=s$. Let $\text{gr}(\cQ(U))_{n,s}\subset \cQ_n(U)/\cQ_{n-1}(U)$, which is spanned
by all elements $a\in \cQ_n(U)/\cQ_{n-1}(U)$ of the form \eqref{bas:a} with $\text{deg}_{\beta }(a) \leq s$. We obtain the following filtration of $\cQ_n(U)/\cQ_{n-1}(U)$:
$$\text{gr}(\cQ(U))_{n,0}\subset \text{gr}(\cQ(U))_{n,1}\subset\cdots\subset  \text{gr}(\cQ(U))_{n,n}=\cQ_n(U)/\cQ_{n-1}(U).$$
Let
$$\psi_{n,s}: \text{gr}(\cQ(U))_{n,s}\to \text{gr}(\cQ(U))_{n,s}/ \text{gr}(\cQ(U))_{n,s-1}$$
be the projection.

Regarding $\text{gr}(\cQ(U))$ as an abelian vertex algebra, $1\in \cQ_{0,0}(U)$ and 
for $a\in \cQ_{n,k}(U)$, $b\in \cQ_{m,l}(U)$, we have
\begin{align*}
&a\circ_n b=\frac{1}{(n-1)!}(\partial^{n-1}a)b&\in \cQ_{n+m,k+l}(U), \quad &n<0;\\
&a\circ_n b=0&\in \cQ_{m+n, k+l-1}(U), \quad &n\geq 0.
\end{align*}
As in Section~\ref{sub.1}, $\{\cQ_{n,s}(U)\}$ is a good increasing filtration of $\text{gr}(\cQ(U))$ and 
the associated graded object
$$\text{gr}^2(\cQ(U))=\bigoplus_{n,s} \text{gr}(\cQ(U))_{n,s}$$
is a $\partial$-ring. Moreover, $\text{gr}^2(Q(U))$ is generated as a $\partial$-ring by $\psi_{1,1}(\beta^i),\psi_{1,0}(b^i), \psi_{0,0}(c^i)$, 
and $\psi_{0,0}(f(\gamma))$, $f(\gamma) \in \mathcal O(U)$. In fact as an abstract  $\partial$-ring, $\text{gr}(\cQ(U))$ is isomorphic to $\text{gr}^2(\cQ(U))$.
When no confusion can arise, the symbols $\beta^i, \gamma^i, b^i, c^i$ in $\text{gr}(\cQ(U))$ will also be used to denote the corresponding elements $\psi_{1,1}(\beta^i)$, $\psi_{0,0}(\gamma^i)$, $\psi_{1,0}(b^i), \psi_{0,0}(c^i)$ in $\text{gr}^2(\cQ(U))$.

 The coordinate transformation equations ~(\ref{chi.co2}) for $\text{gr}(\cQ(U))$ preserve this filtration. 
 So for $U\in \mathcal U$, $U\to \text{gr}(\cQ(U))_{n,s}$ gives a sheaf  $\text{gr}(\cQ)_{n,s}$ on $M$ and $\{\text{gr}(\cQ)_{n,s}\}$ is a filtration of the sheaf 
 $\text{gr}(Q)$ and  the correspondence $U\to \text{gr}^2(\cQ(U))$ with $U\in \mathcal U$, gives a bigraded $\partial$-ring sheaf on the manifold $M$, 
 which is denoted by $\text{gr}^2(\cQ)$. Obviously $$\text{gr}^2(\cQ)=\bigoplus_{n,s} \text{gr}(\cQ)_{n,s}/\text{gr}(\cQ)_{n,s-1}.$$

For the sheaf $\text{gr}^2(\cQ)$, the  relations of $\beta, \gamma, b, c $ under the coordinate transformation are
\begin{align}\label{chi.co3}
\partial^n \tilde \gamma^i&=\partial^{n-1}(\frac{\partial f^i}{\partial \gamma^j}\partial \gamma^j),\nonumber\\
\partial^n \tilde b^i&=\partial^n(\frac{\partial g^i}{\partial \tilde \gamma^j}(g(\gamma))b^j),\nonumber\\
\partial^n \tilde c^i&=\partial^n(\frac{\partial f^i}{\partial \gamma^j}c^j),\\
\partial^n \tilde \beta^i&=\partial^n(\frac{\partial g^i}{\partial \tilde \gamma^j}(g(\gamma))\beta^j).\nonumber
\end{align}
As in the sheaf $\text{gr}(\cQ)$, the derivation $\partial$ is a global operator in $\text{gr}^2(\cQ)$, and satisfies
$$\psi_{n,s}(\partial a)=\partial\psi_{n,s} (a), \quad a\in \text{gr}(\cQ)_{n,s}(U).$$
In particular, the space of the global sections of $\text{gr}^2(\cQ)$ is a $\partial$-ring.

\subsection{Relation to classical bundles}
Consider the following filtration for $\text{gr}(\cQ(U))_{n,s}/\text{gr}(\cQ(U))_{n,s-1}$:
$$\cdots\subset \text{gr}^2(\cQ(U))_{n,s}^{t+1}\subset \text{gr}^2(\cQ(U))_{n,s}^{t}\subset \text{gr}^2(\cQ(U))_{n,s}^{t-1}\cdots,$$
$$\text{gr}(\cQ(U))_{n,s}/\text{gr}(\cQ(U))_{n,s-1}=\bigcup_{m\in \mathbb Z} \text{gr}^2(\cQ(U))_{n,1}^t.$$
$\text{gr}^2(\cQ(U))_{n,s}^{t}$ is spanned by
$$\partial^{k_1}\beta^{i_1}\cdots\partial^{k_s}\beta^{i_s}\partial^{k_{s+1}}b^{i_{s+1}}\cdots \partial^{k_{n}}b^{i_{n}}\partial^{l_1}\gamma^{j_1}\cdots\partial^{l_u}\gamma^{j_u}\partial^{l_{u+1}}c^{j_{u+1}}
\cdot\cdot\cdot \partial^{l_{v}}c^{j_{v}}f, \quad v\geq t.$$
Here $l_s\geq 1$, for $1\leq s\leq u$, $ f\in \mathcal O(U)$.
We also have an associated graded object
$$\text{gr}^3(\cQ(U))=\bigoplus \text{gr}^2(\cQ(U))_{n,s}^{t}/\text{gr}^2(\cQ(U))_{n,s}^{t+1}.$$
It has a $\partial$-ring structure in an obvious way.
 The coordinate transformation equations ~(\ref{chi.co3}) for $\text{gr}^2(\cQ(U))$     
 preserve this filtration $\{\text{gr}^2(\cQ(U))_{n,s}^{t}\}$. So for $U\in \mathcal U$, the correspondence $U\to \text{gr}^2(\cQ(U))_{n,s}^{t}$ gives a sheaf
 $\text{gr}^2(\cQ)_{n,s}^t$ on $M$ and $\{\text{gr}^2(\cQ)_{n,s}^t\}$ is a filtration of the sheaf $\text{gr}(Q)$. Finally, the
 correspondence $U\to \text{gr}^3(\cQ(U))$ with $U\in \mathcal U$, gives a sheaf of $\partial$-rings on
 the manifold $M$, which is denoted $\text{gr}^3(\cQ)$.  Obviously
 $$\text{gr}^3(\cQ)=\bigoplus_{n,s,t} \text{gr}^2(\cQ)_{n,s}^t/\text{gr}^2(\cQ)_{n,s}^{t+1}.$$

 The coordinate transformation formula for the sheaf $\text{gr}^3(\cQ)$ is
                                               \begin{align}
                                 \partial^n \tilde \gamma^i&=\frac{\partial f^i}{\partial \gamma^j}\partial^n \gamma^j ,\nonumber \\
\partial^n \tilde b^i&=\frac{\partial g^i}{\partial \tilde \gamma^j}(g(\gamma))\partial^nb^j ,\nonumber \\
\partial^n \tilde c^i&=\frac{\partial f^i}{\partial \gamma^j}\partial^n c^j ,\\
\partial^n \tilde \beta^i&=\frac{\partial g^i}{\partial \tilde \gamma^j}(g(\gamma))\partial^n\beta^j . \nonumber
\end{align}
So $\text{gr}^3(\cQ)$ is the sheaf corresponding to the vector bundle
\begin{equation} \label{bun.chi3}
\text{Sym}(\bigoplus_{n=1}^{\infty}(TM\bigoplus T^*M))\bigotimes {\bigwedge}^*(\bigoplus_{n=1}^{\infty}(TM\bigoplus T^*M).
\end{equation}

\subsection{Global sections}
In this section, we discuss the relationship between the global sections of $\cQ$, $\text{gr}(\cQ)$ and  $\text{gr}^2(\cQ)$. The exact sequence of sheaves
$$0\to \cQ_{n-1} \to\cQ_{n}\to \cQ_n/\cQ_{n-1}\to 0,$$
induces the exact sequence
\begin{equation}\label{ex:q}
0\to \cQ_{n-1}(M)\to\cQ_{n}(M)\stackrel{\phi_n}{\to} \cQ_{n}/\cQ_{n-1}(M).
\end{equation}
We have the following \textsl{reconstruction property}.

\begin{lemma}\label{lem.q}
If $a_i\in \cQ_{n_i}(M)$, for $1\leq i\leq l$, such that $\phi_{n_i}(a_i)$ generate $\text{gr}(\cQ)(M)$ as a $\partial$-ring, then
$\phi_n$ is surjective, and it therefore induces an isomorphism
$$\cQ_{n}(M)/\cQ_{n-1}(M){\cong} \cQ_{n}/\cQ_{n-1}(M).$$
Furthermore, $\{a_i|\ 1\leq i\leq l\}$ strongly generates the vertex algebra $\cQ(M)$. \end{lemma}

\begin{proof} For the first part of the lemma,
we show that for any $b\in \cQ_n/\cQ_{n-1}(M)$, there is a global section 
$a$ of $\cQ_n$ generated by the $a_i$ under the Wick product and $\partial$, such that $\phi_n(a)=b$.
 Since $\phi_{n_i}(a_i)$ generate $\text{gr}(\cQ)(M)$  and $\phi_n$ is linear, we can assume 
 $$b=\partial^{j_1} \phi_{n_{i_1}}(a_{i_1})\partial^{j_2}\phi_{n_{i_2}}(a_{i_2})\cdots \partial^{j_k}\phi_{n_{i_k}}(a_{i_k}), \quad n=\sum n_i.$$

For  any $a\in \cQ_k(M)$, $b\in\cQ_l(M)$,  we have
$:\partial^sa\partial^t b:\in\cQ_{k+l}(M)$ and
$$\phi_{k+l}(:\partial^sa\partial^t b:)=\partial^s \phi_k(a)\partial^t\phi_l(b).$$
 Let $a=:\partial^{j_1}a_{i_1}\partial^{j_2}a_{i_2}\cdots \partial^{j_k}a_{i_k}:\in \cQ_n(M)$, then
$$\phi_n(a)=\partial^{j_1}\phi_{n_1}(a_{i_1})\phi_{n-n_1}(:\partial^{j_2}a_{i_2}\cdots \partial^{j_k}a_{i_k}:)$$
$$=\cdots
=\partial^{j_1}\phi_{n_{i_1}}(a_{i_1})\partial^{j_2}\phi_{n_{i_2}}(a_{i_2})\cdots \partial^{j_k}\phi_{n_{i_k}}(a_{i_k})=b.$$

For  the second part of lemma, we proceed by induction: for $a\in \cQ_0(M)$, since $\cQ_0$ is a subsheaf of  $\text{gr}(\cQ)$, $\phi_0$ is 
the identification, so $a$ is generated by $a_i$. Assume that for $n\geq 0$,
any global section in $ \cQ_{n}(M)$ is strongly generated by $a_i$.  Now for $a\in\cQ_{n+1}(M)$, 
by the first part of the proof,  there is a global section $\tilde a$ which is generated by $a_i$, such that
$\phi_{n+1}(\tilde a)=\phi_{n+1}(a)$. So $\phi_{n+1}(a-\tilde a)=0 $ and by (\ref{ex:q}) $a-\tilde a \in \cQ_{n}(M)$. 
By assumption, $a-\tilde a$ is strongly generated by $a_i$, so $a$ is strongly generated by $a_i$.  This completes the proof.
\end{proof}

Now we study the relationship between $\text{gr}(\cQ)(M)$ and $\text{gr}^2(\cQ)(M)$.
For $a\in \text{gr}(\cQ)_{n,s}(M)$, $b\in \text{gr}(\cQ)_{m,t}(M)$, we have $$\psi_{n+m,s+t}(\partial^la\partial^kb)=\partial^l\psi_{n,s}(a)\partial^k\psi_{m,t}(b),$$ 
and we have the exact sequence of sheaves
$$0\to \text{gr}(\cQ)_{n,s-1}\to \text{gr}(\cQ)_{n,s}\to \text{gr}(\cQ)_{n,s}/\text{gr}(\cQ)_{n,s_1}\to 0.$$
Using a similar argument to the proof of Lemma \ref{lem.q}, we obtain
\begin{lemma} \label{lem.q1}
If $a_i\in \text{gr}(\cQ)_{n_i,s_i}$ for $1\leq i\leq l$, such that $\psi_{n_i,s_i}(a_i)$ generate $\text{gr}(\cQ)(M)$ as a $\partial$-ring, then
$\psi_{n,s}$ is surjective, and it induces an isomorphism
$$\text{gr}(\cQ)_{n,s}(M)/\text{gr}(\cQ)_{n,s-1}(M){\cong} \text{gr}(\cQ)_{n,s}/\text{gr}(\cQ)_{n,s-1}(M).$$
Furthermore, $\{a_i|\ 1\leq i\leq l\}$ generate $\text{gr}(\cQ)(M)$ as a $\partial$-ring.
\end{lemma}

 \section{Invariants of $\mathfrak{sl_2}[t]$ action}\label{sec.invt}

We work in the standard basis $H,G,F$ for the Lie algebra $\gs\gl_2 =\mathfrak {sl}(2,\mathbb C)$, satisfying
 $$[H,G]=2G,\quad [H,F]=-2F,\quad [G,F]=H.$$
 Let $V=\mathbb Ce^1\oplus \mathbb Ce^2$ be the two-dimensional representation of $\gs\gl_2$  given by
 \begin{align*}
 H e^1&=e^1,\quad He^2=-e^2,\\
 Ge^1&=0, \quad Ge^2=e^1, \\
 Fe^1&=e^2,\quad Fe^2=0 .
 \end{align*}
 Let $V^*= \mathbb C{e^*}^1\oplus \mathbb C{e^*}^2$ be the dual representation. There is an induced representation of $\gs\gl_2[t]$ on the ring
 $$R=\mathbb C[\beta_i^j,b_i^j,c_i^j,\gamma_i^j], \quad i\geq 1,j=1,2.$$ For $\xi\in \gs\gl_2$, $x=\beta^j,\gamma^j, b^j,c^j$, $i=1,2$, we have
 $$\xi t^n(x_m)=(\xi x)_{m-n}, (n\leq m); \quad  \xi t^n(x_m)=0, (n>m).$$ Here $\xi$ acts on $\beta^j,b^j$ as it acts on $e^j$, and $\xi$ acts on $c^j, \gamma^j$ as it acts on ${e^*}^j$. Note that $R$ is a $\partial$-ring with a derivation $D$ given by $$D x_m=m x_{m+1}.$$ Let $R_1=\mathbb C[\beta_1^j,b_1^j,c_1^j,\gamma_1^j]\subset R$. By Corollary 7.2 of \cite{LSSII}, we have the following result.
  \begin{thm} \label{thm.invt}
  $R^{\mathfrak{sl}_2[t]}$ as a $\partial$-ring with derivation $D$ is generated by $R_1^{\mathfrak{sl}_2}$.
    \end{thm}
By Weyl's first and second fundamental theorems of invariant theory for the standard representation of $\gs\gl_2$ \cite{We},  $R_1^{\mathfrak{sl}_2}$ is generated by
  \begin{align}\label{gen.R}
  &\beta_1^1b_1^2-\beta_1^2b_1^1,& &\beta_1^1c_1^1+\beta_1^2c_1^2,&&\beta_1^1\gamma_1^1+\beta_1^2\gamma_1^2,&
  &b_1^1c_1^1+b_1^2c_1^2,\nonumber \\
  &b_1^1\gamma_1^1+b_1^2\gamma_1^2,& &c_1^1\gamma_1^2-c_1^2\gamma_1^1,&
  &b_1^1b_1^2,& & c_1^1c_1^2.
    \end{align}
 So  as a $\partial$-ring with the derivation $D$, $R^{\mathfrak{sl}_2[t]}$ is generated by (\ref{gen.R}). In particular, as a supercommutative ring, $R^{\mathfrak{sl}_2[t]}$ is generated by the $l$-th derivatives of the generators \eqref{gen.R}, for $l\geq 0$.

 Now let
 \begin{align}
 \partial^ix^j&=i! x_{i+1}^j,\quad i\geq 0,\,x=\beta, b, c;  \nonumber    \\
 \partial^{i}\gamma^j&={(i-1)!}\gamma^j_i, \quad \quad i\geq 1  ;   
 \end{align}
 Then $$R=\mathbb C[\partial^l\beta^i,\partial^{l+1}\gamma^i,\partial^lb^i,\partial^lc^i],\quad l\geq 0,\, i=1,2.$$
 For $\xi\in\mathfrak g$, $x=\beta^i, \partial\gamma^i, b^i, c^i$, We have
 \begin{align}\xi t^n(\partial^mx)&=\frac {m!}{(m-n)!}\partial^{m-n}(\xi x), & m\geq n,  \nonumber \\
  \xi t^n(\partial^mx)&=0, &m<n.
 \end{align}

In this set of generators, we have $D\partial^nx=\partial^{n+1}x$, for $x=\beta^i, \partial\gamma^i, b^i, c^i$. 
So we can replace the derivation $D$ by the symbol $\partial$.
The generators for  $R^{\mathfrak{sl}_2[t]}$  in $(\ref{gen.R})$ are

  \begin{align}\label{gen.R2}
  &A_{\beta b}=\beta^1b^2-\beta^2b^1,& &A_{\beta c}=\beta^1c^1+\beta^2c^2,&&A_{\beta\partial \gamma}=\beta^1\partial\gamma^1+\beta^2\partial\gamma^2,&\nonumber\\
  &A_{bc}=b^1c^1+b^2c^2, &&A_{\partial\gamma b}=b^1\partial \gamma^1+b^2\partial\gamma^2,& &A_{\partial \gamma c}=c^1\partial\gamma^2-c^2\partial\gamma^1,&\\
  &  A_{bb}=b^1b^2,& & A_{cc}=c^1c^2.\nonumber 
    \end{align}

\subsection{A standard monomial theory for $R^{\mathfrak{sl}_2[t]}$}\label{sec.mon} We will use the standard monomial theory to give a basis of $R^{\mathfrak{sl}_2[t]}$.  

As a super commutative ring, $S=\{\partial^kA_{xy}|\ A_{xy}  \text{ is  one of  ~(\ref{gen.R2}})\}$ is a set of generators of $R^{\mathfrak{sl}_2[t]}$. We give an ordering to the symbols $\beta, \partial\gamma, b, c$
$$\beta>\partial\gamma>b>c.$$
The partial ordering of $S$ is given by
$\partial^kA_{xy}<\partial^{k'}A_{x'y'}$  if and only if one of the following three cases is satisfied:
\begin{enumerate}
\item $k\leq k'-2 $ ;
\item  $k=k'-1$ and  \begin{enumerate}\item if $x'$ is odd, $x'>y$,
   								\item if $x'$ is even, $x'\geq y$;
                                       \end{enumerate}

\item   $k=k'$ and \begin{enumerate}
					\item if $x', y'$ are even,  $x'> x$, $ y'\geq y$ or  $x'\geq  x$, $y'> y$ ,
					\item if $x'$ is even and $y'$ is odd, $x'\geq x$ and  $y'>y$,
					\item if $x'$ is odd and $y'$ is even, $x'> x$ and  $y'\geq y$,
					\item if $x'$ is odd and $y'$ is even, $x'> x$ and  $y'>y$.
				\end{enumerate}
\end{enumerate}
An ordered product $s_{1}s_{2}\cdots s_{k}$	of elements of $S$ is said to be \textsl{standard} if  $s_{i}\leq s_{i+1}$, for $1\leq i< k$ and the equality can be taken only when $s_{i}=\partial^kA_{\beta\partial \gamma}$, $k\geq 0$. 

 \begin{thm}\label{thm.standard}
 $R^{\mathfrak{sl}_2[t]}$ has a standard monomial theory for $S$, i.e. the standard monomials forms a basis of $R^{\mathfrak{sl}_2[t]}$.
 \end{thm}
 \begin{proof} We first prove that any ordered product of elements of $S$ is a linear combination of standard monomials. Since $S$ generates $R^{\mathfrak{sl}_2[t]}$, any element of $R^{\mathfrak{sl}_2[t]}$ is a linear combination of standard monomials.
 It is convenient, for what follows, to  order the elements of $S$ by
  $$\partial^kA_{xy} \prec\partial^{k'}A_{x'y'},\text{ if } k<k' \text{ or }k=k' \text{ and } x<x'  \text{ or } k=k' , x=x' \text{ and } y< y' ,$$
  and then order the words (ordered product of elements) of $S$  lexicographically.  From the definitions of the ordering and the partial orderings of  $S$, we can see that $\partial^kA_{xy}<\partial^{k'}A_{x'y'}$ imply $\partial^kA_{xy} \prec\partial^{k'}A_{x'y'}$.
 
 We have the following quadratic relations for elements in (\ref{gen.R2}):
 \begin{align}
 A_{xy}A_{zz}=0,& &\text{ if } z=b \text{ or } c \text{ and } z=x \text{ or } y;&\label{eq.s1}\\
 A_{xy}A_{zz}+\alpha A_{xz}A_{yz}=0,& &\text{ if } z=b\text{ or } c \text{ and } z\neq x,y, \text{ here } A_{cb} \text{ means } A_{bc},&\label{eq.s2}\\
 && \alpha\text{ is a nonzero constant};\nonumber\\
 A_{\beta\partial \gamma}A_{bc}-A_{\beta b}A_{\partial \gamma c}+A_{\beta c}A_{\partial \gamma b}=0.\label{eq.s3}
\end{align}
Applying $\partial^{2k}$ and $\partial^{2k+1}$, $k\geq 0$ on these relations,  we get quadratic relations for elements  in $S$. For example, for $z=b$ or $c$ and $z=x$ or $y$, applying $\partial^{2k}$ on ~(\ref{eq.s1}), we get
$$
\frac{(2k)!}{k!k!}\partial^k A_{xy}\partial^kA_{zz}=-\sum_{i=1}^k \frac{(2k)!}{(k-i)!(k+i)!}(\partial^{k-i} A_{xy}\partial^{k+i}A_{zz}+ \partial^{k-i} A_{zz}\partial^{k+i}A_{xy}).$$

Multiplying by the constant $\frac{k!k!}{(2k)!}$ on both sides, observe that the monomials on the right hand side of the above equation are standard and are smaller than the monomial on left hand side.  We use '$s.s.$' to represent them and we get
\begin{align}
\underline{\partial^k A_{xy}\partial^kA_{zz}}=s.s.. \label{eq.s11}
\end{align}
Similarly, applying $\partial^{2k+1}$ on ~(\ref{eq.s1}), we get
\begin{align}\label{eq.s12}
&\underline{\partial^k A_{xc}\partial^{k+1}A_{cc}}+\partial^kA_{cc}\partial^{k+1}A_{xc}
=s.s. &\text{       for } x=\beta,\partial\gamma, b;\nonumber\\
&2\underline{\partial^k A_{cc}\partial^{k+1}A_{cc}}
=s.s.; &\nonumber\\
&\underline{\partial^k A_{xb}\partial^{k+1}A_{bb}}+\partial^kA_{bb}\partial^{k+1}A_{xb}
=s.s. &\text{       for } x=\beta,\partial\gamma;\\
&\partial^k A_{bc}\partial^{k+1}A_{bb}+\underline{\partial^kA_{bb}\partial^{k+1}A_{bc}}
=s.s. ;&\nonumber\\
&2\underline{\partial^k A_{bb}\partial^{k+1}A_{bb}}=s.s..\nonumber
\end{align}
Applying $\partial^{2k}$ on ~(\ref{eq.s2}), we get
\begin{align}\label{eq.s21}
&\partial^kA_{zz}\partial^k A_{\beta\partial\gamma}+\alpha\underline{\partial^k A_{\partial\gamma z}\partial^kA_{\beta z}}=s.s.; & \text{ for } z=b,c;\nonumber\\
&\underline{\partial^kA_{bb}\partial^k A_{x c}}+\alpha\partial^k A_{bc}\partial^kA_{xb}=s.s.; & \text{ for } x=\beta, \partial \gamma;\\
&\partial^kA_{cc}\partial^k A_{xb}+\alpha\underline{\partial^k A_{bc}\partial^kA_{xc}}=s.s.; & \text{ for } x=b,\beta,\partial\gamma;\nonumber
\end{align}
Applying $\partial^{2k+1}$ on ~(\ref{eq.s2}), we get
\begin{align}\label{eq.s22}
&\partial^kA_{zz}\partial^{k+1} A_{\beta\partial\gamma}
+\underline{\partial^{k} A_{\beta\partial\gamma}\partial^{k+1}A_{zz}}
+\alpha\partial^k A_{\partial\gamma z}\partial^{k+1}A_{\beta z}
-\alpha\partial^k A_{\beta z}\partial^{k+1}A_{\partial\gamma z}=s.s.; & \text{ for } z=b,c;\nonumber\\
&\partial^kA_{bb}\partial^{k+1} A_{xc}
+\partial^{k} A_{xc}\partial^{k+1}A_{bb}
+\alpha\partial^k A_{bc}\partial^{k+1}A_{xb}
+\alpha\underline{\partial^k A_{xb}\partial^{k+1}A_{bc}}=s.s.; & \text{ for } x=\beta, \partial \gamma;\\
&\partial^kA_{cc}\partial^{k+1} A_{xb}
+\underline{\partial^{k} A_{xb}\partial^{k+1}A_{cc}}
+\alpha\partial^k A_{bc}\partial^{k+1}A_{xc}
+\alpha\partial^k A_{xc}\partial^{k+1}A_{bc}=s.s.;&\text{ for } x=b,\beta,\partial\gamma;\nonumber
\end{align}
And from ~\ref{eq.s3} we get
\begin{align}
\partial^kA_{bc}\partial^k A_{\beta\partial \gamma}-\partial^kA_{\partial \gamma c}\partial^kA_{\beta b}+\underline{\partial^kA_{\partial \gamma b}\partial^kA_{\beta c}}=s.s.;\label{eq.s31}
\end{align}
\begin{align}\label{eq.s32}
\underline{\partial^k A_{\beta\partial \gamma}\partial^{k+1}A_{bc}}+\partial^kA_{\beta b}\partial^{k+1}A_{\partial \gamma c}-\partial^kA_{\beta c}\partial^{k+1}A_{\partial \gamma b}\\
+\partial^kA_{bc}\partial^{k+1} A_{\beta\partial \gamma}-\partial^kA_{\partial \gamma c}\partial^{k+1}A_{\beta b}+\partial^kA_{\partial \gamma b}\partial^{k+1}A_{\beta c}=s.s..\nonumber
\end{align}
 The underlined monomials  are the largest  in each of  the above equations \eqref{eq.s11}-\eqref{eq.s32}, and the other monomials are standard.  One can check that all of the ordered words like $\partial^kA_{xy}\partial^{k'}A_{x'y'}$ with $\partial^kA_{xy}\preceq\partial^{k'}A_{xy}$  and $k'\leq k+2$, if they are not equal to zero, are standard except the underlined monomials in \eqref{eq.s11}-\eqref{eq.s32}.  Since when $k'\geq k+2$, $\partial^kA_{xy}\partial^{k'}A_{x'y'}$ is standard. Thus

\textsl{Any ordered word of the form $\partial^kA_{xy}\partial^{k'}A_{x'y'}$ with $\partial^kA_{xy}\preceq\partial^{k'}A_{xy}$ that is not standard, can be written as a linear combination of  standard monomials, which are  lexicographically smaller.}

 Now  for any ordered  word $s=s_{1}s_{2}\cdots s_{n}$ with $s_i=\partial^{k_i}A_{x_ky_k}$,  if $s_i\succ s_{i+1}$, switching $s_i$ and $s_{i+1}$, we get a lexicographically smaller word. If $s_i\preceq s_{i+1}$ and $s_is_{i+1}$ is not standard, by the above result, $s_is_{i+1}$ can be written as a linear combination of   lexicographically smaller words. So $s$ is equal to a linear combination of lexicographically smaller words of $S$. Thus 
 
 \textsl{Any ordered word of $S$ that is not standard, can be written as a linear combination of lexicographically smaller words of $S$.}
 
 Using the above statement,  for any ordered word of $S$, if it is not standard, we can write it as a a linear combination of standard words of $S$, which are lexicographically smaller.

Now we prove that the set of the standard monomials is linearly independent.

We order  the generators of $R$ by
$$\partial^kc_1<\partial^kc_2<\partial^kb_2<\partial^kb_1<\partial^{k+1}\gamma_1<\partial^{k+1}\gamma_2<\partial^k\beta_2<\partial^k\beta_1<\partial^{k+1}c_1<\partial^{k+1}c_2.$$
Then we order the monomials of $R$ lexicographically, i.e. for two monomials of $R$, which are written in the form $r=r_1\cdots r_k$ and $r'=r_1'\cdots r_{k'}'$ with $r_i\leq r_{i+1}$ and $r_i'\leq r_{i+1}'$,
$$r<r' \text{ if and only if (1) }   k<k' \text{ or  (2) }  k=k' ,  r_i=r_i' \text{ for } i<j \text{ and }  r_j<r_j'.$$
For any element $s\in R$, let $L(s)$ be the largest monomial in $s$. 

For example, let  $ a=1$ for $x= \beta, b$, $ a=2$ for $ x =\partial \gamma, c$; $\bar a=1$ for $y=\partial \gamma, c$, $\bar a=2$ for $y=\beta, b$.
$$
L(\partial ^{2k}A_{xy})=C_1\partial^ky_a\partial^kx_{\bar a}, \quad L(\partial ^{2k+1}A_{xy})=C_2\partial^kx_a\partial^{k+1}y_{\bar a}.
$$
Here $C_1, C_2$ are nonzero constants. Thus for $s , s'\in S$,  $L(s)$ and  $L(s')$ can be written in the form $C x_ay_{\bar a}$ and $C'x'_ay'_{\bar a}$, respectively. If $ss'$ is standard, by the definition of the partial ordering, we have
$$x_a\leq x'_a,  \text{ for } x_a \text{ even, } ( x_a< x'_a,\text{ for } x_a \text{ odd});$$
$$y_{\bar a}\leq y'_{\bar a},  \text{ for } y_{\bar a} \text{ even, } ( y_{\bar a}< y'_{\bar a},\text{ for } y_{\bar a} \text{ odd}).$$
And $L(ss')=CC'x_ay_{\bar a}x'_ay'_{\bar a}.$
In fact, $L$ satisfies the property:
$$L(ss')=L(s)L(s'),\quad\text{ for } s,s'\in R \text{ and }L(s)L(s')\neq 0 .$$ 
 Thus if $s=s_1\cdots s_n$ is a standard monomial and $L(s_i)=x^i_ay^i_{\bar a}$, then
 $$L(s)=L(s_1)\cdots L(s_n)=C x^1_ay^1_{\bar a}\cdots x^n_ay^n_{\bar a}, \quad C\neq 0$$
 and if $s$ and $s'$ are two standard monomials, then $s\neq s'$ if and only if $L(s)\neq  L(s')$ . So 
 $\{L(s)| s \text{ is a standard monomial of } S\}$ are linearly independent. $L$ is a linear map, so the set of standard monomials of $S$ is linearly independent. We have proved  that the standard monomials of $S$ form a basis of $R^{\mathfrak{sl}_2[t]}$.
  
  \end{proof} 
\subsection{A criterion for $\mathfrak{sl}_2[t]$ invariance}
We need the following lemma later to determine whether elements of $R$ are $\mathfrak{sl}_2[t]$ invariant.
\begin{lemma}\label{lem:in}
If $f\in R$  is $\mathfrak{sl}_2$ invariant and satisfies
\begin{equation}\label{eq:in}
\sum_{k\geq 1}(Ht^kf) \frac{\partial^k
\gamma^1}{k!}+
\sum_{k\geq 1}(Gt^k f)\frac{\partial^k \gamma^2}{k!}=0 ,
\end{equation}
then $f$ is $\mathfrak{sl}_2[t]$ invariant.
\end{lemma}
\begin{proof} Letting $G$ act on both sides of (\ref{eq:in}), we get
\begin{equation}\label{eq:inH}
-2\sum_{k\geq 1}(Gt^kf) \frac{\partial^k
\gamma^1}{k!}-
\sum_{k\geq 1}(Ht^kf) \frac{\partial^k \gamma^2}{k!}=0.
\end{equation}
Letting $G$ act on the both sides of the above equation, we get
\begin{equation}\label{eq:inG}
4\sum_{k\geq 1}(Gt^kf) \frac{\partial^k
\gamma^2}{k!}=0.
\end{equation}
Letting $F$ act on (\ref{eq:inG}), we get
\begin{equation}\label{eq:inF}
\sum_{k\geq 1}(Ht^kf) \frac{\partial^k
\gamma^2}{k!}+\sum_{k\geq 1}(Gt^kf) \frac{\partial^k
\gamma^1}{k!}=0.
\end{equation}
Then from (\ref{eq:inH}) and (\ref{eq:inF}), we get
$$\sum_{k\geq 1}(Gt^kf) \frac{\partial^k
\gamma^1}{k!}=0.$$
Denote $$\mathbb O_1=\sum_{k\geq 1}\frac{\partial^k
\gamma^1}{k!} Gt^k ,\quad \mathbb O_2=\sum_{k\geq 1}\frac{\partial^k
\gamma^2}{k!} Gt^k .$$
We have $\mathbb O_1f=\mathbb O_2f=0$. Let
$$\mathbb P_n=[\cdots[[\mathbb O_2,\mathbb O_1],\underbrace{\mathbb O_1]\cdots,\mathbb O_1]}_n,\quad n\geq 0.$$
Then $\mathbb P_n f=0$. We calculate this operator as follows.
\begin{align*}\mathbb P_n
&=[\cdots[\sum_{k\geq 1}\frac{\partial^k
\gamma^2}{k!}\sum_{l\geq 1} (Gt^k\frac{\partial^l\gamma^{1_1}}{l_1!})Gt^{l_1},\underbrace{\mathbb O_1]\cdots,\mathbb O_1]}_n\\
&=[\cdots[\sum_{k\geq 1}\frac{\partial^k
\gamma^2}{k!}\sum_{l_1> k} (\frac{\partial^{l_1-k}\gamma^2}{l_1(l_1-1-k)!})Gt^{l_1},\underbrace{\mathbb O_1]\cdots,\mathbb O_1]}_n\\
&=\sum_{k\geq 1}\frac{\partial^k
\gamma^2}{k!}\sum_{l_1> k} (\frac{\partial^{l_1-k}\gamma^2}{l_1(l_1-1-k)!})\sum_{l_2> l_1} 
(\frac{\partial^{l_2-l_1}\gamma^2}{l_2(l_2-1-l_1)!})\cdots \sum_{l_{n+1}> l_n} 
(\frac{\partial^{l_{n+1}-l_n}\gamma^2}{l_{n+1}(l_{n+1}-1-l_n)!}) Gt^{l_{n+1}}\\
&=\frac{(\partial\gamma^2)^{n+2}}{(n+2)!}Gt^{n+2}+\sum_{l\geq n+3}a_lGt^l.
\end{align*}
Notice that there is some $N$, such that when $n\geq N$, $Gt^n f=0$. Let $N_0$ be the minimal integer with this property. If $N_0\geq 3$,
$$0=\mathbb P_{N_0-3}f=\frac{(\partial\gamma^2)^{n+2}}{(n+2)!}Gt^{N_0-1}f+\sum_{l\geq N_0}a_lGt^l 
f=\frac{(\partial\gamma^2)^{n+2}}{(n+2)!}Gt^{N_0-1}f.$$
Thus $Gt^{N_0-1}f=0$, which contradicts the minimality of $N_0$.
So $Gt^2f=0$. Thus $0=O_1f=\gamma^1 Gt f$, we conclude $Gt f=0$. Since $Gt$ and $\mathfrak{sl}_2$ generate the 
Lie algebra $\mathfrak{sl}_2[t]$, $f$ is $\mathfrak{sl}_2[t]$ invariant.
\end{proof}

  \section{Global sections of the chiral de Rham complex on Kummer surfaces}\label{sec.global}
 \subsection{Chiral de Rham complex on Kummer surfaces}
  Let $X$ be a Kummer surface. Then there is an abelian surface
   $A=\mathbb C^2/\Lambda$ with $\Lambda\cong \mathbb Z^4$ and an involution $\iota: A\to A$ defined 
   by $\iota(a)=-a$. The quotient $\tilde X=A/{\iota}$ has $16$  ordinary double points $\{p_i\}$. 
   $X$ can be constructed by blowing up at these singular points. Let $\pi: X\to \tilde X$, $C_i=\pi^{-1}(p_i)$ are genus zero curves.
  Let $\tilde z^1, \tilde z^2$ be coordinates of $\mathbb C^2$, $\tilde z^1, \tilde z^2$ can be
   regarded as  coordinates on $X\backslash{\{C_i\}}$. Assume $p_0=(0,0)$, we can use two charts to 
   cover a neighborhood $C_0$: $(U_1, z^1,  z^2), (U_2, \bar z^1, \bar z^2)$. Their relations with $\tilde z^1, \tilde z^2$ are
  $$z^1=(\tilde z^1)^2, \quad  z^2= \frac{\tilde z^2}{\tilde z^1};$$
  $$\bar z^1=(\tilde z^2)^2, \quad \bar z^2= \frac{\tilde z^1}{\tilde z^2}.$$
  These relations give  relations between $\tilde \beta, \partial\tilde \gamma, \tilde b, \tilde c$
  and $\beta, \partial\gamma, b, c$ of the chiral de Rham complex $\cQ=\cQ_X$ on $X$.
  Then by the equations (\ref{chi.co3}), we get the corresponding relations for $\beta, \partial\gamma, b, c$ 
  of the refined associated graded sheaf $\text{gr}^2(Q)$ on $X$:  
  \begin{align}   \label{kum.coo}
  \partial^{l+1} \tilde \gamma^1&=\partial^l(\frac t 2 \partial \gamma^1)
  =\frac t 2\partial^{l+1}\gamma^1-\frac{t^3}4\sum_{k=1}^l\partial^{l-k}(\partial\gamma^1\partial^k\gamma^1)+O(t^5);   \nonumber  \\
  \partial^{l+1}\tilde \gamma^2&=\partial^{l}(\frac 1 t\partial \gamma^2+\frac t 2 \gamma^2  \partial \gamma^1)\nonumber\\
  &=t\partial^{l+1}\gamma^2+\frac{t\gamma^2}2\partial^{l+1}\gamma^1
  +\frac t 2 \sum_{k=1}^l\partial^{l-k}(\partial\gamma^1\partial^k\gamma^2)+\frac t 2 \sum_{k=1}^l\partial^{l-k}(\partial\gamma^2\partial^k\gamma^1)+O(t^3);\nonumber \\
  \partial^l\tilde c^1&=\partial^l(\frac t 2 c^1)=\frac t 2\partial^{l}c^1-\frac{t^3}4\sum_{k=0}^{l-1}
  \partial^{l-k-1}(\partial\gamma^1\partial^k c^1)+O(t^5);           \nonumber \\
  \partial^l\tilde c^2&=\partial^l(\frac 1 tc^2+\frac t 2 \gamma^2  c^1)\nonumber\\
  &=t\partial^{l}c^2+\frac{t\gamma^2}2\partial^{l}c^1
  +\frac t 2 \sum_{k=0}^{l-1}\partial^{l-k-1}(\partial\gamma^1\partial^kc^2)+\frac t 2 \sum_{k=0}^{l-1}
  \partial^{l-k-1}(\partial\gamma^2\partial^kc^1)+O(t^3); 
  \end{align}
   \begin{align}
  \partial^l\tilde b^1&=\partial^l(\frac 2 tb^1-t  \gamma^2  b^2)     \nonumber \\
  &=\frac 2 t\partial^{l}b^1-t\gamma^2\partial^{l}b^2
  +t\sum_{k=0}^{l-1}\partial^{l-k-1}(\partial\gamma^1\partial^kb^1)-t \sum_{k=0}^{l-1}\partial^{l-k-1}
  (\partial\gamma^2\partial^kb^2)+O(t^3);    \nonumber \\
  \partial^l\tilde b^2&= \partial^l(t  b^2) = 
  t\partial^{l}b^2-\frac{t^3}2\sum_{k=0}^{l-1}\partial^{l-k-1}(\partial\gamma^1\partial^k b^2)+O(t^5);           \nonumber  
  \end{align} 
  \begin{align}
  \partial^l\tilde \beta^1&=\partial^l(\frac 2 t \partial^l\beta^1-t\gamma^2\beta^2)\nonumber \\
  &=\frac 2 t\partial^{l}\beta^1-t\gamma^2\partial^{l}\beta^2
  +t\sum_{k=0}^{l-1}\partial^{l-k-1}(\partial\gamma^1\partial^k\beta^1)-t 
  \sum_{k=0}^{l-1}\partial^{l-k-1}(\partial\gamma^2\partial^k\beta^2)+O(t^3);    \nonumber \\
  \partial^l\tilde\beta^2&=\partial^l(t\beta^2)=t\partial^{l}\beta^2-
  \frac{t^3}2\sum_{k=0}^{l-1}\partial^{l-k-1}(\partial\gamma^1\partial^k \beta^2)+O(t^5) .   \nonumber \end{align}
  Here $t=\frac 1{\sqrt{\gamma_1}}$, $\partial^{l} t=-\frac 1 2t^3\partial^l \gamma^{1}+O(t^5).$

 \subsection{Global sections of $\text{gr}^2({\cQ})$}
 We introduce the following notation:  \\
 when $(x,y)=(\beta,\partial\gamma),(\beta,c),(b,\partial \gamma),(b,c),
 (\tilde \beta,\partial\tilde \gamma),(\tilde \beta,\tilde c),(\tilde b,\partial \tilde \gamma),(\tilde b,\tilde c)$,
 let
 $$A(x,y,i,j)=\partial^ix^1\partial^jy^1+\partial^ix^2\partial^jy^2,$$
 when $(x,y)=(\beta,\beta),(\beta,b),(b,b),(c,c),(\partial \gamma,\partial \gamma),(c,\partial \gamma),
 (\tilde\beta,\tilde\beta),(\tilde\beta,\tilde b),(\tilde b,\tilde b),(\tilde c,\tilde c),
 (\partial \tilde\gamma,\partial\tilde \gamma),(\tilde c,\partial \tilde\gamma)$,
 let
 $$A(x,y,i,j)=\partial^ix^1\partial^jy^2-\partial^ix^2\partial^jy^1.$$
 
 By identifying the elements $\beta, \partial \gamma, b, c \in R$ in Section \ref{sec.invt} with those in $\text{gr}^2(\cQ)(U_1)$, $R$ can be regarded as a subring of $\text{gr}^2(\cQ)(U_1)$. 
The generators of $R_1^{\mathfrak{sl}_2}$ in (\ref{gen.R2}) are
 $$A_{\beta\partial\gamma}=A(\beta,\partial \gamma,0,0),\quad A_{bc}=A(b,c,0,0),
 \quad A_{\beta c}=A(\beta, c,0,0), \quad A_{\partial\gamma b}=A(\partial \gamma,b,0,0),$$
 $$A_{cc}=c_1c_2,\quad A_{bb}=b_1b_2,
 \quad A_{\beta b}=A(\beta,b,0,0),\quad A_{\partial \gamma c}=A(\partial\gamma,c,0,0).$$
Since $\partial$ is a global operator, 
the $\partial$-ring generated by these elements are in $\text{gr}^2(\cQ)(X)$.
We want to prove 
\begin{thm} \label{thm.glo}
 $A_{\beta\partial\gamma}, A_{bc}, A_{\beta c}, A_{\partial\gamma b}, A_{cc},A_{bb} ,A_{\beta b},A_{\partial\gamma c}$ 
 generate $\text{gr}^2(\cQ)(X)$ as a $\partial$-ring.

 \end{thm}
  For any $p \in U_1$, the composition 
 $$R\to \text{gr}^2(\cQ)(U_1)\to \text{gr}^2(\cQ)(U_1)|_p$$ is an isomorphism.
 Since $A_{\beta\partial\gamma}, A_{bc}, A_{\beta c}$, $ A_{\partial\gamma b}, A_{cc},A_{bb} ,A_{\beta b},A_{\partial\gamma c}$  generate $R^{\mathfrak{sl}_2[t]}$,  we have
 
\begin{cor} 
\label{thm.inv}$$\text{gr}^2(\cQ)(X)\cong  R^{\mathfrak{sl}_2[t]} \cong \text{gr}^2(\cQ)|_p^{\mathfrak{sl}_2[t]}.$$
 \end{cor}
By Theorem~\ref{thm.standard}, $R^{\mathfrak{sl}_2[t]}$ has a standard monomial theory for $S$, so $\text{gr}^2(\cQ)(X)$ has a standard monomial theory. In particular,  it has a canonical basis.

  Now we prove the theorem. If $g$ is a global section of $\text{gr}^2(\cQ)$, then $\pi_*(g|_{X\backslash \{C_i\}})$ is defined on $\tilde X\backslash\{p_i\}$, which can be pulled back to $A$, and extended to a global section on $A$. For the abelian variety $A$, in coordinates $(\tilde z_1,\tilde z_2)$,
   $$\tilde R=\text{gr}^2(\cQ)(A)=\mathbb C[\partial^k\tilde\beta^i,\partial^{k+1}\tilde \gamma^i,\partial^k\tilde b^i,
    \partial^k\tilde c^i],
   \,i=1,2,\,k\geq 0.$$
  So in the coordinates $(\tilde z^1,\tilde z^2)$, we naturally have $\text{gr}^2(\cQ)(X)\subset \tilde R$.

  From \cite{Ko}, the global sections of tensors of tangent bundles and cotangent bundles
 on a  $K3$ surface  are parallel and  $G=SU(2)$ invariant at a point on the surface. Here $G$ is the
 holonomy group of $K3$. It acts naturally on the fibres of tangent bundles and cotangent bundles. 
 For a fixed point $q\in X\backslash \{C_i\}$, $\tilde R\cong \tilde R|_q$ is  isomorphic to  the fibre of the bundle (\ref{bun.chi3}) at q. 
 Then we have
 \begin{lemma}\label{lem:gl3}
 For the Kummer surface $X$
 $$\text{gr}^3(\cQ)(X)\cong \text{gr}^3(\cQ)(X)|_p= \tilde R|_p^{SU(2)}\cong \tilde R^{SU(2)}. $$
 \end{lemma}
  $A(\tilde x,\tilde y,i,j)$, $x,y=\beta,\partial \gamma, b,c$ and $i,j\geq 0$
 are global sections of $\text{gr}^3(\cQ)$. On the other hand, by the first fundamental theorem of invariant theory for $SU(2)$, they generate the ring $\tilde R^{SU(2)}$. Thus by Lemma~\ref{lem:gl3}, we have
 \begin{lemma}\label{lem:gg}
    $A(\tilde x,\tilde y,i,j)$, $x,y=\beta,\partial \gamma, b,c$ and $i,j\geq 0$ generate $\text{gr}^3(\cQ)(X).$
 \end{lemma}

 If $F$ is a global section of $\text{gr}^2(\cQ)_{n,s}$,  in  the coordinates $(\tilde z_1,\tilde z_2)$,  $F $ can be written as a polynomial
 in $\tilde \beta, \partial\tilde\gamma,\tilde b, \tilde c$ and their $\partial$-derivatives
 with homogeneous degree in $\partial^l b$, $l\geq 0$ being $n-s$
  and homogeneous degree in $\partial^l \beta$, $l\geq 0$ being $s$.
 As a polynomial in $\partial^{l+1}\gamma,\partial^l c $, $l\geq 0$,
  $$F=f_a+f_{a+1}+F_{a+2},$$ where every monomial of $F_{a+2}$ has degree at least $a+2$,
    $f_a\neq 0$ is homogeneous of degree $a$ and
 $f_{a+1}$ is homogeneous of degree $a+1$.  So as polynomials in  
 $\partial^l\tilde \beta, \partial^{l+1}\tilde\gamma,\partial^l\tilde b,\partial^l \tilde c,\, l\geq 0$,
 $f_a$ is homogeneous of degree $n+a$, $f_{a+1}$ is homogeneous of degree $n+a+1$, every
 monomial of $F_{a+2}$ has degree at least $n+a+2$.

Through the morphism of sheaves
$$H_a: \text{gr}^2(\cQ)_{n,s}^a\to \text{gr}^2(\cQ)_{n,s}^a/\text{gr}^2(\cQ)_{n,s}^{a+1}\subset \text{gr}^3(\cQ).$$
   $H_a(F)$ is a global section of $\text{gr}^3(\cQ)$, and we have
   $$H_a(F)=f_a(H_0(\tilde\beta),H_1(\tilde\partial \gamma),H_0(\tilde b),H_1(\tilde c)).$$
   Thus by Lemma \ref{lem:gg}, $f_a$ is generated by $A(\tilde x,\tilde y,i,j)$.

   From the coordinate transformation formula (\ref{kum.coo}), we have
   $$\partial^k\tilde x=\sum a_n t^{2i+1},  \text{ for }x=\beta, \partial\gamma, b ,c,\, k\geq 0.$$
   Here $a_n$ is a polynomial of $\partial^l\beta,\partial^{l}\gamma,\partial^l b,\partial^l c,\,l\geq 0$.
    So the  coefficient of $t^{2k}$ will be zero.
    Since $f_a\neq 0$, $f_a$ is generated by the quadratics $A(x,y,i,j)$, $a+n$ must be even, i.e. $n+a+1$ must be odd. Then
    $$f_{a+1}(\tilde \beta,\tilde \gamma,\tilde b, \tilde c)=\sum a_n(\beta,\gamma,b,c)t^{2n+1}.$$

 The coordinate transformation formula for $A(b,c, i,j)$ is
 $$A(\tilde b,\tilde c,i,j)=A(\tilde b,\tilde c,i,j)+\frac {t^2} 2 \big(-\partial^ib^1\sum_{k=0}^{j-1}\partial ^{j-k-1}(\partial \gamma^1\partial^kc^1)+\partial^ib^2\sum_{k=0}^{j-1}\partial^{j-k-1}(\partial \gamma^1\partial^kc^2)\big)$$
 $$+\frac {t^2} 2 \big(\sum_{k=0}^{i-1}\partial^{i-k-1}(\partial \gamma^1\partial^kb^1)\partial^jc^1-
 \sum_{k=0}^{i-1}\partial^{i-k-1}(\partial \gamma^1\partial^kb^2)\partial^jc^2\big)$$
 $$+\frac {t^2} 2 \big(-\sum_{k=0}^{i-1}\partial^{i-k-1}(\partial \gamma^2\partial^kb^2)\partial^jc^1+
 \partial^ib^2\sum_{k=0}^{i-1}\partial^{i-k-1}(\partial \gamma^2\partial^kc^1)\big)+O(t^4)$$
 $$=A(b,c,i,j)+\frac {t^2} 2\big( \sum_{k\geq 1}Ht^k(A(b,c,i,j)) \frac{\partial^k \gamma^1}{k!}+
 \sum_{k\geq 1}Gt^k(A(b,c,i,j)) \frac{\partial^k \gamma^2}{k!}\big)
 +O(t^4).$$
 In fact one can check that for $x,y=\beta, \partial\gamma, b, c,$
 $$A(\tilde x,\tilde y,i,j)=\alpha \bigg(A(x,y,i,j)+
\frac {t^2} 2\big( \sum_{k\geq 1}Ht^k(A(x,y,i,j)) \frac{\partial^k \gamma^1}{k!}+ \sum_{k\geq 1}Gt^k(A(x,y,i,j)) \frac{\partial^k \gamma^2}{k!}\big)\bigg)+O(t^4).$$
  Here $\alpha$  is a constant.
  \begin{enumerate}
\item $\alpha=0, \quad (x,y)=(\beta,c),(\beta,\partial \gamma),(b,c),(b,\partial \gamma);$
\item $\alpha=2,\quad (x,y)=(\beta,\beta),(b,b),(\beta,b);$
\item $\alpha=\frac 1 2,\quad (x,y)=(\partial \gamma,\partial\gamma),(\partial\gamma,c),(c,c).$
  \end{enumerate}

 Thus
 $f_a( \tilde\beta, \partial \tilde\gamma,\tilde b,\tilde c)$
 $$
 ={ 2}^{\frac{n-a}2}\bigg(f_a(\beta, \partial\gamma, b, c)
  +\frac {t^2} 2\big( \sum_{k\geq 1}Ht^kf_a(\beta, \partial\gamma, b, c) \frac{\partial^k \gamma^1}{k!}+
  \sum_{k\geq 1}Gt^kf_a(\beta, \partial\gamma, b, c) \frac{\partial^k \gamma^2}{k!}\big) \bigg)
  +O(t^4).$$

  From the coordinate transformation formula (\ref{kum.coo}), we know that the degree of $\partial^lc,\partial^{l+1}\gamma$,
  $l\geq 0$ will not decrease after the coordinate transformation. Since every monomial of $F_{a+2}$ has
  degree of $\partial^l\tilde c,\partial^{l+1}\tilde \gamma$ at least $a+2$, after changing coordinates, the degree of $\partial^l\tilde c,\partial^{l+1}\tilde \gamma$ will still be at least $a+2$.
Then
$$
F(\tilde \beta,\partial \tilde \gamma,\tilde b,\tilde c)=f_a(\tilde \beta,\partial \tilde \gamma,\tilde b,\tilde c)+f_{a+1}+F_{a+2}$$
  $$= \alpha^{\frac{n-i}2}\bigg(f_i(\beta, \partial\gamma, b, c)
      +\frac {t^2} 2\big( \sum_{k\geq 1}Ht^kf_a(\beta, \partial\gamma, b, c) \frac{\partial^k \gamma^1}{k!}+
      \sum_{k\geq 1}Gt^kf_a(\beta, \partial\gamma, b, c) \frac{\partial^k \gamma^2}{k!}\big) \bigg)$$
 $$      +\text{ terms } t^2 \text{ with degree of }\partial^{l+1}\gamma,\partial^lc{\geq a+2} +\text{ other } t \text{ terms }.$$
 $t^2$ is singular at the chart $(U_1, z^1,z^2)$. So if $F$ is a global section on $X$, we have

$$\sum_{k\geq 1}Ht^kf_a(\beta, \partial\gamma, b, c) \frac{\partial^k
\gamma^1}{k!}+
\sum_{k\geq 1}Gt^kf_a(\beta, \partial\gamma, b, c) \frac{\partial^k \gamma^2}{k!}=0 . $$

By Lemma \ref{lem:in}, we conclude that $f_a$ is $\mathfrak{sl}_2[t]$ invariant. So it can be generated by $A_{\beta\partial\gamma}$, $A_{bc}$, $A_{\beta c}$, $A_{\partial\gamma b}$, $A_{cc}$, $A_{bb}$, $A_{\beta b}$, $A_{\partial\gamma c}$ and their $l$th derivatives, for $l\geq 0$. Then $f_a$ is a global section of $\text{gr}^2({\cQ})$. Since $F$ is a global section, then $F-f_a$ is a global section. By induction on $a$, we see that $F$ is $\mathfrak{sl}_2[t]$ invariant. This completes the proof of Theorem \ref{thm.glo} and  Corollary \ref{thm.inv}.

\subsection{Global sections of chiral de Rham complex on Kummer surface}
Consider the sections in $\cQ(U_1)$:
\begin{align}\label{sec:gl}
 &Q(z)= :\beta^1(z)c^1(z):+:\beta^2(z)c^2(z):,&\nonumber\\
 &L(z)=\sum_{i=1}^2(:\beta^i(z)\partial\gamma^i(z):-:b^i(z)\partial c^i(z):),&\nonumber\\
&J(z)=-:b^1(z)c^1(z):-:b^2(z)c^2(z):,&\nonumber\\
&G(z)=:b^1(z)\partial\gamma^1(z):+:b^2(z)\partial\gamma^2(z):,&\\
&B(z)= :\beta^1(z)b^2(z):-:\beta^2(z)b^1(z):,&\nonumber\\
 &D(z)= :b^1(z)b^2(z):,& \nonumber \\
&C(z)=:\partial \gamma^1(z)c^2(z):-:\partial\gamma^2(z)c^1(z):,&\nonumber\\
 &E(z)=:c^1(z)c^2(z):.&\nonumber
\end{align}
 These  sections can be extended to global sections of $\cQ$ on $X$. Note that $L(z), J(z), Q(z), G(z)$ are global sections on any Calabi-Yau manifold, and $E(z)$ and $D(z)$ correspond to the nowhere vanishing holomorphic $2$-form on $X$ and its dual respectively. The last two sections come from $$C(z)=G\circ_0 E(z), \quad B=Q\circ_0 D(z).$$

Now we can prove
\begin{thm}\label{thm.main}The eight global sections given by (\ref{sec:gl}) strongly generate $\cQ(X)$.
\end{thm}
\begin{proof}Locally, the images of the eight  global sections given by (\ref{sec:gl}) in $\text{gr}^2(Q)(X)$ are
\begin{align*}
&\psi_{2,0}(\phi_2(D(z)))=A_{bb}, &\psi_{2,1}(\phi_2(B(z)))=A_{\beta b},\\
&\psi_{1,1}(\phi_1(L(z)))=A_{\beta\partial\gamma},&\psi_{1,1}(\phi_1(Q(z)))=A_{\beta c},\\
&\psi_{1,0}(\phi_1(J(z)))=-A_{bc}, &\psi_{1,0}(\phi_1(G(z)))=A_{b\partial\gamma},\\
&\psi_{0,0}(\phi_0(C(z)))=A_{c\partial\gamma},&\psi_{0,0}(\phi_0(E(z)))=A_{cc}.
\end{align*}
By Theorem \ref{thm.glo}, $\psi_{2,0}(\phi_2(D(z)))$, $\psi_{2,1}(\phi_2(B(z)))$, $\psi_{1,1}(\phi_1(L(z)))$,  
$\psi_{1,1}(\phi_1(Q(z)))$, $\psi_{1,0}(\phi_1(J(z)))$, $\psi_{1,0}(\phi_1(G(z)))$,
  $\psi_{0,0}(\phi_0(C(z)))$, $\psi_{0,0}(\phi_0(E(z)))$,  generate $\text{gr}^2(\cQ)(X)$ as a $\partial$-ring. So by Lemma \ref{lem.q1},
  $\phi_2(D(z))$, $\phi_2(B(z))$, $\phi_1(L(z))$, $\phi_1(Q(z))$, $\phi_1(J(z))$, $\phi_1(G(z))$, 
  $\phi_0(C(z))$, $\phi_0(E(z))$ generate $\text{gr}(\cQ)(X)$ as a $\partial$-ring.
  Then by Lemma \ref{lem.q}, $D(z)$, $B(z)$, $L(z)$, $Q(z)$, $ J(z)$, $G(z)$, $C(z)$, $E(z)$ strongly generate $\cQ(X)$.
\end{proof}

\begin{cor} For a Kummer surface $X$, $\cQ(X)$ is isomorphic to the $N=4$ superconformal algebra with central charge $c=6$.
\end{cor}
\begin{proof} The generators and OPE relations of this algebra can be found on page 187 of \cite{Kac}, and checking this is a straightforward calculation. Note that 
the Virasoro element $L(z)$ above must be replaced with $L(z) -1/2 \partial J(z)$, which has central charge $6$. 
\end{proof}

\begin{remark} In fact, $\cQ(X)$ is isomorphic to the {\it simple} $N=4$ vertex algebra with $c=6$. This follows from a more general result which will appear in a separate paper.
\end{remark}

Finally, we use the basis of $R^{\mathfrak{sl}_2[t]}$ constructed in Section~\ref{sec.mon} to construct a basis of $\cQ(X)$.
For convenience,  let $\alpha$ be the 'inverse' of $\psi\phi$ on the set $S$:
\begin{align}
&\alpha(\partial^k A_{bb})=\partial^kD(z),&&\alpha(\partial^k A_{\beta b})=\partial^kB(z),&&\alpha(\partial^k A_{\beta\gamma})=\partial^kL(z),&\nonumber\\
&\alpha(\partial^k A_{\beta c})=\partial^kQ(z),&&\alpha(\partial^k A_{bb})=\partial^kJ(z),&&\alpha(\partial^k A_{b\partial\gamma})=\partial^kG(z),&\nonumber\\
&\alpha(\partial^k A_{c\partial\gamma})=\partial^kC(z),&&\alpha(\partial^k A_{cc})=\partial^kE(z).&\nonumber
\end{align} 
For a standard monomial $s=s_1s_2\cdots s_k$ of $S$, let 
$$\alpha(s)=:\alpha(s_1)\alpha(s_2)\cdots\alpha (s_k):.$$ 
Then $\psi_*(\phi_{*,*}(\alpha(s)))=s$, with appropriate subscripts of $\psi$ and $\phi$.  It is easy to check that
the set 
$$\{\alpha(s)|s \text{ is a standard monomial of }S\}$$
is a linear basis of $\cQ(X)$.


\begin{thebibliography}{ABKS}
\bibitem{BHS}  D. ~Ben-Zvi, R.~ Heluani,  M. ~Szczesny, \textit{Supersymmetry of the chiral de Rham complex}, Compos. Math. 144 (2008) 503-521.
\bibitem{B} R. Borcherds, \textit{Vertex operator algebras, Kac-Moody algebras and the monster}, Proc. Nat. Acad. Sci. USA 83 (1986) 3068-3071.
\bibitem{Bor} L. Borisov, \textit{Vertex algebras and mirror symmetry}, Commun. Math. Phys. 215 (2001) 517-557.
 \bibitem{FLM} I.B. Frenkel, J. Lepowsky, and A. Meurman, \textit{Vertex Operator Algebras and the Monster}, Academic Press, New York, 1988.
 \bibitem{H} R. Heluani, \textit{Supersymmetry of the chiral de Rham complex II: commuting sectors}, Int. Math. Res. Not. 2009 (2009) no. 6, 953-987.
 \bibitem{LiI} H. Li, \textit{Local systems of vertex operators, vertex superalgebras and modules}, J. Pure Appl. Algebra 109 (1996), no. 2, 143-195.
 \bibitem{LiII} H.~Li, \textit{Vertex algebras and vertex Poisson algebras}, Commun. Contemp. Math. 6 (2004) 61-110.
 \bibitem{LL} B.~Lian, A.~Linshaw, \textit{Howe pairs in the theory of vertex algebras}, J. Algebra 317,111-152 (2007).
 \bibitem{Kac} V. Kac, \textit{Vertex Algebras for Beginners}, University Lecture Series, Vol. 10. American Math. Soc., 1998.
 \bibitem{Ko}S.~Kobayshi, \textit{First Chern class and holomorphic tensor fields}, Nagoya Math. J. 77(1980), 5-11.

 \bibitem{Kap} A.~Kapustin, \textit{Chiral de Rham complex and the half-twisted sigma-model}. arXiv:hep-th/0504074.
 \bibitem{LZ} B. Lian and G.J. Zuckerman, \textit{New perspectives on the BRST-algebraic structure of string theory}, Comm. Math Phys. 154 (1993) 613-646.
  \bibitem{LSSI}A.~Linshaw, G.~Schwarz, B.~Song, \textit{Jet schemes and invariant theory}, arXiv:1112.6230.
 \bibitem{LSSII}A.~Linshaw, G.~Schwarz, B.~Song, \textit{Arc spaces and the vertex algebra commutant problem}, arXiv:1201.0161.
 \bibitem{MSV}F.~Malikov, V.~Schectman, and A.~Vaintrob, \textit{Chiral de Rham complex}, Comm. Math. Phys.  204, (1999) 439-473.
 \bibitem{MS}F.~Malikov, V.~Schectman, \textit{Chiral de Rham complex. II},  Differential topology, infinite-dimensional Lie algebras, and applications, 149-188, Amer. Math. Soc. Transl. Ser. 2, 194, Amer. Math. Soc., Providence, RI, 1999.
 \bibitem{We} H. Weyl, \textit{The Classical Groups: Their Invariants and Representations}, Princeton University Press, 1946.



\end{thebibliography}
\end{document}